\documentclass[a4paper,12pt,reqno]{amsart}

\usepackage{amsmath}
\usepackage{amssymb}
\usepackage{amsfonts}
\usepackage{graphicx}
\usepackage{mathtools}
\usepackage[colorlinks]{hyperref}
\renewcommand\eqref[1]{(\ref{#1})} 

\graphicspath{ {images/} }
\title[Improved Discrete Hardy inequality]{Improvement of the Discrete Hardy inequality}

\author[P. Roychowdhury]{Prasun Roychowdhury}
\address{
	Prasun Roychowdhury
	\endgraf
	Mathematics Division
    \endgraf
    National Center for Theoretical Sciences
	\endgraf
	National Taiwan University
    \endgraf
    No. 1, Sec. 4, Roosevelt Road, Taipei 10617
	\endgraf
	Taiwan
	\endgraf
	{\it E-mail address} {\rm prasunroychowdhury1994@gmail.com}}

\author[D. Suragan]{Durvudkhan Suragan}
\address{
	Durvudkhan Suragan
	\endgraf
	Department of Mathematics
	\endgraf
	Nazarbayev University
	\endgraf
	Kazakhstan
	\endgraf
	{\it E-mail address} {\rm durvudkhan.suragan@nu.edu.kz}}

\date{\today}

\theoremstyle{plain}

\newtheorem{theorem}{Theorem}[section]
\newtheorem{proposition}{Proposition}[section]
\newtheorem{lemma}{Lemma}[section]

\newtheorem{remark}{Remark}[section]

\def\N{{I\!\!N}}

\numberwithin{equation}{section} \allowdisplaybreaks

\usepackage[text={6in,8.6in},centering]{geometry}
\newcommand{\n}{{\mathbb{N}}}

\begin{document}
 	\begin{abstract}
       We establish a novel improvement of the classical discrete Hardy inequality, which gives the discrete version of a recent (continuous) inequality of Frank, Laptev, and Weidl. Our arguments build on certain weighted inequalities based on discrete analogues of symmetric decreasing rearrangement techniques.    
	\end{abstract}
	\maketitle
	
\section{Introduction}
In a letter dated 21 June 1921, E.~Landau communicated to G.~H.~Hardy by including the proof of an inequality which read as: Let $1<p<\infty$ and $\{a_n\}_{n=1}^{\infty}$ be a sequence of nonnegative real numbers, then there holds
\begin{align}\label{hardy}
	  \sum_{n=1}^{N}\bigg(\frac{1}{n}\sum_{k=1}^{n}a_{k}\bigg)^p\leq \bigg(\frac{p}{p-1}\bigg)^p	\sum_{n=1}^{N}a_n^p	
\end{align}
for all positive integers $N\in\n:=\{1,2,\cdots\}$ or $N=\infty$. Moreover, the constant $\big(\frac{p}{p-1}\big)^p$ is the best constant in the sense that the inequality \eqref{hardy} does not hold if one replaces the constant on the right-hand side with some smaller one. Also if the right-hand side of \eqref{hardy} is finite, then the equality case holds if and only if $\{a_n\}_{n=1}^{\infty}$ is the identically zero sequence.

Afterward, the inequality \eqref{hardy} was independently proved by Hardy \cite{hardy} when he was trying to find a simple proof of the famous \emph{Hilbert inequality} related to the convergence of double series sum. Nowadays, in the literature, the inequality \eqref{hardy} is known as the discrete version of \emph{Hardy's inequality} or sometimes it is called the \emph{Hardy-Landau inequality}. We refer to \cite{lef-1} for a short and direct proof of the discrete Hardy inequality. Other mathematicians of the beginning of the 20th century such as G.~P\'olya, I.~Schur, and M.~Riesz also contributed to the development of Hardy’s and related inequalities during the period 1906-1928 and we refer to \cite{kuf} for an interesting historical survey.

Let $C_c(\n)$ be the space of all finitely supported functions on $\n$. Then for $\psi\in C_c(\n)$, one defines the first order difference operator by
\begin{align*}
    \nabla \psi(n):=\psi(n)-\psi(n-1) \quad \text{ for all }n\geq 1,
\end{align*}
where $\psi$ satisfies the Dirichlet boundary condition $\psi(0)=0$. 

Let $1<p<\infty$. For all $\psi\in C_c(\n)$ with $\psi(0)=0$ there holds 
\begin{align}\label{grad-hardy}
 \sum_{n=1}^{\infty}\frac{|\psi(n)|^p}{n^p}\leq 	  \bigg(\frac{p}{p-1}\bigg)^p	\sum_{n=1}^{\infty}|\nabla \psi(n)|^p,
\end{align}
with the sharp constant $\big(\frac{p}{p-1}\big)^p$  and equality holds only when $\psi\equiv 0$. It is worth mentioning that \eqref{grad-hardy} can be directly derived from \eqref{hardy}. 

There is an extensive literature on various improvements of the continuous analogues of Hardy’s inequality \eqref{grad-hardy}, but compared to that little is known about these inequalities in the discrete setting. Recently, for the case $p=2$, the inequality \eqref{grad-hardy} has been improved by several authors (see e.g. \cite{GKS, hng, kpm-2, kls22, kjc}). Among those, in \cite{kpm-1} it was shown that by using criticality theory arguments on graphs for discrete Schr\"odinger operators one can obtain the best possible improvement in the case $p=2$. In addition, here we should mention the very recent work \cite{gupta1d}, where the author considered the power weight to improve \eqref{grad-hardy} significantly. For the case $p\neq 2$, in the literature, there are only a few studies which consider improvements of \eqref{grad-hardy}. An improved version of the classical discrete Hardy inequality \eqref{grad-hardy} was obtained for $1 < p < \infty$ in \cite{fkp}. Without claim of completeness, we also mention the references \cite{fish, ll, kpm-3, lef-2} and \cite[Chapter 7]{lif} for the interested reader to revisit recent progress in the theory of the discrete Hardy inequality.

In this paper, we prove 
	\begin{align}\label{imp-nograd-hardy}
		  \sum_{n=1}^{\infty}\sup_{0< m <\infty}\bigg|\min\biggl\{\frac{1}{n}\,,\,\frac{1}{m}\biggr\}\sum_{k=1}^{m}\psi(k)\bigg|^p \leq \bigg(\frac{p}{p-1}\bigg)^p	\sum_{n=1}^{\infty}\big|\psi(n)\big|^p
	\end{align}
and
	\begin{align}\label{imp-grad-hardy}
		\sum_{n=1}^{\infty}\max\biggl\{\sup_{0< m \leq n}\frac{|\psi(m)|^p}{n^p},\: \sup_{n\leq  m <\infty} \frac{|\psi(m)|^p}{m^p}\biggr\} \leq  \bigg(\frac{p}{p-1}\bigg)^p \sum_{n=1}^{\infty}\big|\nabla \psi(n)\big|^p
	\end{align}
for a given $1<p<\infty$ and any $\psi\in C_c(\mathbb{N})$ (with the Dirichlet boundary condition $\psi(0)=0$ in
\eqref{imp-grad-hardy}).
It is straightforward to verify that 
\begin{align}\label{sharper1}
  \sum_{n=1}^{\infty}\bigg|\frac{1}{n}
  \sum_{k=1}^{n}\psi(k)\bigg|^p\leq     \sum_{n=1}^{\infty}\sup_{0< m <\infty}\bigg|\min\biggl\{\frac{1}{n}\,,\,\frac{1}{m}\biggr\}
  \sum_{k=1}^{m}\psi(k)\bigg|^p ,
\end{align}
and 
\begin{align}\label{sharper2}
  \sum_{n=1}^{\infty}\frac{|\psi(n)|^p}{n^p}\leq   \sum_{n=1}^{\infty}\max\biggl\{\sup_{0< m \leq n}\frac{|\psi(m)|^p}{n^p},\: \sup_{n\leq  m <\infty} \frac{|\psi(m)|^p}{m^p}\biggr\} ,
\end{align}
which give improvements of \eqref{hardy} and \eqref{grad-hardy}, correspondingly. We believe that these results give new insight into the theory of discrete Hardy's inequalities without having a restriction of the power $1<p<\infty$. The main inspiration of this work comes from the very recent progress in the field derived by R.~L.~Frank, A.~Laptev, and T.~Weidl \cite{rupert-22}, and in their paper, the authors study the continuous versions of the above inequalities. To the best of our knowledge, for $1<p<\infty$ their version of improved Hardy's inequality is completely new in the literature and it can be also called the Frank-Laptev-Weidl improvement of the Hardy inequality \cite{rrs-22}.

Let us recall the result of Fischer, Keller, and Pogorzelski (see \cite[Theorem 1]{fkp}) on an improvement of the discrete $p$-Hardy inequality in the arithmetic summation form. We believe that for the discrete setting, up to now, it is the best possible improvement.

Let $1<p<\infty$. Then for all $\psi\in C_c(\n)$ there holds
\begin{align}\label{best-p-hardy}
    \sum_{n=1}^{\infty}v_p(n)\bigg|\sum_{k=1}^{n}\psi(k)\bigg|^p \leq \sum_{n=1}^{\infty}|\psi(n)|^p,
\end{align}
where $v_p$ is a strictly positive function given by
\begin{align*}
    v_p(n)=\bigg(1-\left(1-\frac{1}{n}\right)^{\frac{p-1}{p}}\bigg)^{p-1}-\bigg(\left(1+\frac{1}{n}\right)^{\frac{p-1}{p}}-1\bigg)^{p-1}
\end{align*}
with
\begin{align*}
    v_p(n)>\bigg(\frac{p-1}{p}\bigg)^p\frac{1}{n^p}, \quad n\in\n.
\end{align*}
Moreover, for any integer $p\geq2$ one has
\begin{equation}\label{vpn} v_p(n)=\sum_{l\in 2\n\cup\{0\}}
c_{l}\,n^{-l-p} 
\end{equation}
with $c_{l}>0$ and interchanging the order of summation we have
\begin{align}\label{best-hardy}
    \sum_{l\in 2\n\cup\{0\}}c_{l}\sum_{n=1}^{\infty}\frac{1}{n^{l+p}}\bigg|\sum_{k=1}^{n}\psi(k)\bigg|^p \leq \sum_{n=1}^{\infty}|\psi(n)|^p.
\end{align}
In Theorem \ref{best-hardy-imp} we give a sharper version of the inequality \eqref{best-hardy}.

In general, for all $p>1$ our technique implies a sharper version of the inequality \eqref{best-p-hardy} under the assumption $c_{l}\geq0$ in the weight function series expansion \eqref{vpn}.
Note that in \cite{fkp} the authors conjecture that all these coefficients are strictly positive for all $p>1$.
The obtained inequalities in the present paper can be also extended to more general settings, for example, on graphs (cf. \cite{kpm-2021} and \cite{kpm-1}).

This short paper has a simple structure. Section \ref{2} is the main toolbox of the paper. We construct a new sequence of non-increasing terms from the original sequence in such a way that they belong to the same sequence space $\ell^p(\n)$ with the same norm. In Section \ref{3}, we present the main results of this paper concerning new improvements of discrete Hardy's inequality. In Section \ref{4}, we give some immediate consequences of the obtained results including the so-called uncertainty principle for the discrete datum.

\section{Preliminaries}\label{2}
In this section, we discuss preliminary constructions which are the main tools for the proofs. Let $\psi\in C_c(\mathbb{N})$. We define a non-increasing function from $\psi(n)$. Thus, for all $n\in\mathbb{N}$ the new function $\tilde{\psi}(n)$ is given by the formula
\begin{equation*}
		\tilde{\psi}(n) :=
		\begin{dcases}
			\max\{|\psi(k)|\}_{k=1}^\infty & \text{if } n=1, \\
			\max\biggl\{\{|\psi(k)|\}_{k=1}^\infty \setminus \{\tilde{\psi}(k)\}_{k=1}^{n-1}\biggr\}  & \text{if } n\geq 2.\\
		\end{dcases}
	\end{equation*}
In the above definition, the numeration is repeated in this series according to its multiplicity. 

For clarity let us demonstrate the following simple example. Consider the sequence function
\begin{align*}
    \{u(n)\}_{n=1}^\infty=\{-4,3,3,-3,7,7,0,0,\cdots\}.
\end{align*}
Then by the above definition, we have
\begin{align*}
    \tilde{u}(1)=\max\{4,3,3,3,7,7,0,0,\cdots\}=7,
\end{align*}
\begin{align*}
    \tilde{u}(2)&=\max\left\{\{4,3,3,3,7,7,0,0,\cdots\}\setminus\{7\}\right\}\\&=
    \max\{4,3,3,3,7,0,0,\cdots\}=7,
\end{align*}
and in the same way one finds $\tilde{u}(3)=4$, $\tilde{u}(4)=3$, $\tilde{u}(5)=3$, $\tilde{u}(6)=3$, and $\tilde{u}(6+k)=0$ for $k\geq 1$.

By definition, the function $\tilde{\psi}(n)$ is non-increasing with respect to its variable $n$. Thus, $\tilde{\psi}(1)$ is the largest value of $|\psi|$, $\tilde{\psi}(2)$ is the second largest value of $|\psi|$ (considering the repetition) and so on. Note that this construction coincides with the symmetrization arguments on one-dimensional integer lattice (see \cite{gupta-sym}).

Before mentioning the important properties of the newly constructed sequence, let us first recall the sequence space $\ell^p(\n)$. For $1\leq p<\infty$, we say a sequence $\{a_n\}_{n=1}^{\infty}\in\ell^p(\n)$ if 
\begin{align*}
    \sum_{n=1}^{\infty}|a_n|^p<\infty.
\end{align*}
It is well known that $(\ell^p(\n),||\cdot||_p)$ is a Banach space and its norm can be defined by 
\begin{align*}
   ||a_n||_p:= \bigg(\sum_{n=1}^{\infty}|a_n|^p\bigg)^{1/p}.
\end{align*}

Now we are in a position to present some basic properties of the constructed sequence $\tilde{\psi}(n)$.
\begin{lemma}\label{com-lem}
Let $1\leq p<\infty$ and $\psi\in C_c(\mathbb{N})$. Then  $\{\psi(n)\}_{n=1}^{\infty}$ is an element of the sequence space $\ell^p(\mathbb{N})$ and we have:
\begin{itemize}
    \item[(i)] $\{\tilde{\psi}(n)\}_{n=1}^{\infty}$ becomes an element of $\ell^p(\mathbb{N})$ with the same norm, i.e., there holds
    \begin{align}\label{norm}
    \sum_{n=1}^\infty|\psi(n)|^p=\sum_{n=1}^\infty|\Tilde{\psi}(n)|^p.
\end{align}
\item[(ii)] There also holds
\begin{align}\label{length}
    \sum_{n=1}^m|\psi(n)|^p\leq\sum_{n=1}^m|\Tilde{\psi}(n)|^p,\quad m\in\n.
\end{align}
\end{itemize}
\end{lemma}
\begin{proof}
Since $\psi\in C_c(\n)$, it terminates after finitely many terms and by the construction  $\tilde{\psi}$ terminates after finitely many terms as well. Therefore, both $\psi$ and $\tilde{\psi}$ belong to the space $\ell^p(\n)$. It is clear that these sequences are rearrangements of each other. Thus,  we have \eqref{norm}. 

For the second claim, let us fix $m\in \n$. Then the sequence $\{|\psi(n)|\}_{n=1}^{m}$ can be arranged non-increasingly as in the above way. As a result, we get $\{\tilde{v}(n)\}_{n=1}^m$. Now for any $n\leq m$, there holds
\begin{align*}
    \tilde{v}(n)&= n^{th} \text{ largest element of }\biggl\{|\psi(1)|,|\psi(2)|,\cdots,|\psi(m)|\biggr\}\\& \leq n^{th} \text{ largest element of }\biggl\{|\psi(1)|,|\psi(2)|,\cdots,|\psi(m)|, |\psi(m+1)|,\cdots\biggr\}=\tilde{\psi}(n).
\end{align*}
It implies
\begin{align*}
    \sum_{n=1}^m|\psi(n)|^p=\sum_{n=1}^m|\tilde{v}(n)|^p\leq\sum_{n=1}^m|\Tilde{\psi}(n)|^p,
\end{align*}
which is the desired estimate.
\end{proof}

A simple observation on weighted positive non-increasing sequences gives the following statement.

\begin{lemma}\label{lemma-avg-w}
Let $\{a_n\}_{n=1}^\infty$ be a non-increasing sequence of nonnegative real numbers. Assume $w:\n\rightarrow \mathbb{R}$ is some strictly positive non-decreasing weight function with the property that for any $n,m\in\n$ with $n\leq m$ there holds $m w(n)\leq n w(m)$. Then we have
\begin{align}\label{avg-w}
\frac{1}{w(m)}\sum_{k=1}^{m}a_k\leq \frac{1}{w(n)}\sum_{k=1}^{n}a_k
\end{align}
for all $n\leq m$.
\end{lemma}
\begin{proof}
 $\{a_n\}_{n=1}^{\infty}$ is non-increasing sequence, that is, $a_n\geq a_{n+1}$ for all $n\geq1$. So, for $n\leq m$, we have
\begin{align*}
  &\frac{1}{w(n)}\sum_{k=1}^{n}a_k-\frac{1}{w(m)}\sum_{k=1}^{m}a_k\\
  &=\frac{1}{w(n)w(
  m)}  \left((w(m)-w(n))\sum_{k=1}^{n}\:a_k-w(n)\sum_{k=n+1}^{m}\:a_k \right)\\
  &\geq \frac{1}{w(n)w(m)}  \left(\:(w(m)-w(n))\,n\,a_n-(m-n)\:w(n)\:a_{n+1}\right)\\&\geq \frac{1}{w(n)w(m)}  \left(\:(w(m)-w(n))n\:a_n-w(n)(m-n)\:a_{n}\right)\\
  &=\frac{a_n}{w(n)w(m)}\:\left(\:nw(m)-w(n)n-mw(n)+nw(n)\right)\\&=\frac{a_n}{w(n)w(m)}\:\left(\:nw(m)-mw(n)\right)\ge 0.
\end{align*}
At the end, non-negativity follows from the assumption on the weight function $w$.
\end{proof}

\begin{remark}
Note that by setting $w(s)=s$ for $s\in \n$ in Lemma \ref{lemma-avg-w} we get
\begin{align}\label{avg}
\frac{1}{m}\sum_{k=1}^{m}a_k\leq \frac{1}{n}\sum_{k=1}^{n}a_k 
\end{align}
for any $n,m\in\n$ with $n\leq m$.
\end{remark}

\section{Main Results}\label{3}
Now we state the main results of this paper. First, the improved version of the inequality \eqref{hardy} is proved. Then the new improvement of discrete Hardy's inequality \eqref{imp-grad-hardy} is derived. 

\begin{proposition}\label{key-prop}
	Let $1\leq p<\infty$ and $\psi\in C_c(\n)$. Assume $w:\n\rightarrow \mathbb{R}$ is some strictly positive non-decreasing weight function such that $m w(n)\leq n w(m)$ for any $n,m\in\n$ with $n\leq m$. Then there holds
	\begin{align}\label{key-prop-eqn}
	\sup_{0<m<\infty}\bigg|\min\biggl\{\frac{1}{w(n)},\frac{1}{w(m)}\biggr\}\sum_{k=1}^{m}\psi(k)\bigg|^p  \leq \left( \frac{1}{w(n)}\sum_{k=1}^{n}|\Tilde{\psi}(k)|\right)^{p}.
	\end{align}
\end{proposition}
\begin{proof}
We split the proof into two cases.

{\bf Case 1:} Let $0<m\leq n$. Then we have
	\begin{align*}
		\bigg|\min\biggl\{&\frac{1}{w(n)},\frac{1}{w(m)}\biggr\}\sum_{k=1}^{m}\psi(k)\bigg|\leq\frac{1}{w(n)}\sum_{k=1}^{m}|\psi(k)|\\&\overset{\eqref{length}}{\leq} \frac{1}{w(n)}\sum_{k=1}^{m}|\Tilde{\psi}(k)|\leq\frac{1}{w(n)}\sum_{k=1}^{n}|\Tilde{\psi}(k)|.
	\end{align*}
	
{\bf Case 2:} Let $n\leq m<\infty$. Then exploiting the non-increasing nature of $\Tilde{\psi}$, we obtain
	\begin{align*}
		\bigg|\min\biggl\{&\frac{1}{w(n)},\frac{1}{w(m)}\biggr\}\sum_{k=1}^{m}\psi(k)\bigg|\leq\frac{1}{w(m)}\sum_{k=1}^{m}
		|\psi(k)|\\&\overset{\eqref{length}}{\leq}\frac{1}{w(m)}\sum_{k=1}^{m}
		|\tilde{\psi}(k)|\overset{\eqref{avg-w}}{\leq}\frac{1}{w(n)}\sum_{k=1}^{n}|\Tilde{\psi}(k)|. 
	\end{align*}
	It follows from both the cases that
	\begin{align*}
		\bigg|\min\biggl\{\frac{1}{w(n)},\frac{1}{w(m)}\biggr\}\sum_{k=1}^{m}\psi(k)\bigg|\leq \frac{1}{w(n)}\sum_{k=1}^{n}|\Tilde{\psi}(k)|.
	\end{align*}
This gives the estimate \eqref{key-prop-eqn}.
\end{proof}

Now we are in a position to state one of our main results which improves the original discrete Hardy inequality \eqref{hardy}.
\begin{theorem}\label{mainthm-1}
	Let $1<p<\infty$ and $\psi\in C_c(\n)$. Then there holds
	\begin{align*}
		  \sum_{n=1}^{\infty}\sup_{0< m <\infty}\bigg|\min\biggl\{\frac{1}{n}\:,\:\frac{1}{m}\biggr\}\sum_{k=1}^{m}\psi(k)\bigg|^p \leq \bigg(\frac{p}{p-1}\bigg)^p	\sum_{n=1}^{\infty}\big|\psi(n)\big|^p.
	\end{align*}
Moreover, the constant $\big(\frac{p}{p-1}\big)^p$ is sharp.
\end{theorem}
\begin{proof}
Let us consider the weight function $w(n)=n$ for all $n\geq 1$ in Proposition \ref{key-prop}, so we obtain  
\begin{align*}
    \sup_{0< m <\infty}\bigg|\min\biggl\{\frac{1}{n}\:,\:\frac{1}{m}\biggr\}\sum_{k=1}^{m}\psi(k)\bigg|^{p}\leq \left(\frac{1}{n}\sum_{k=1}^{n}|\Tilde{\psi}(k)|\right)^{p}
\end{align*}
for all $n\ge 1$. 

Therefore, taking the summation over $n$ we have
	\begin{align*}
		&\sum_{n=1}^{\infty}\sup_{0<m<\infty}\bigg|\min\biggl\{\frac{1}{n},\frac{1}{m}\biggr\}\sum_{k=1}^{m}\psi(k)\bigg|^p \\& \leq \sum_{n=1}^{\infty} \left( \frac{1}{n}\sum_{k=1}^{n}|\Tilde{\psi}(k)|\right)^{p}\\&\leq \bigg(\frac{p}{p-1}\bigg)^p
		\sum_{n=1}^{\infty}|\Tilde{\psi}(n)|^{p}\\&\overset{\eqref{norm}}{=}\bigg(\frac{p}{p-1}\bigg)^p
		\sum_{n=1}^{\infty}|\psi(n)|^{p}.
	\end{align*}
Here we have used the classical discrete Hardy inequality \eqref{hardy}. Now the inequality \eqref{sharper1} implies sharpness of the constant since $\big(\frac{p}{p-1}\big)^p$  is sharp for the inequality  \eqref{hardy}.
\end{proof}

The following theorem can be derived from  Theorem \ref{mainthm-1}. We provide its short proof.
\begin{theorem}\label{mainthm-2}
	Let $1<p<\infty$ and $\psi\in C_c(\mathbb{N})$ with $\psi(0)=0$. Then there holds
	\begin{align*}
	\sum_{n=1}^{\infty}\max\biggl\{\sup_{0< m \leq n}\frac{|\psi(m)|^p}{n^p},\: \sup_{n\leq  m <\infty} \frac{|\psi(m)|^p}{m^p}\biggr\}
		\leq \bigg(\frac{p}{p-1}\bigg)^p \sum_{n=1}^{\infty}\big|\nabla\psi(n)\big|^p.
	\end{align*} 
Moreover, the constant $\big(\frac{p}{p-1}\big)^p$ is sharp.
\end{theorem}
\begin{proof}
Let $\psi\in C_c^\infty$ with $\psi(0)=0$. Define a new function $\phi$:
\begin{align*}
    \phi(n):=\psi(n)-\psi(n-1)=\nabla\psi(n) \quad \text{for all}\quad n\geq 1. 
\end{align*}
Thus, we have $\phi\in C_c^\infty(\mathbb{N})$. Now applying Theorem \ref{mainthm-1} for $\phi$, we deduce
\begin{align*}
	 \sum_{n=1}^{\infty}\sup_{0< m <\infty}\bigg|\min\biggl\{\frac{1}{n}\:,\:\frac{1}{m}\biggr\}\sum_{k=1}^{m}\phi(k)\bigg|^p\leq 	\bigg(\frac{p}{p-1}\bigg)^p \sum_{n=1}^{\infty}\big|\phi(n)\big|^p.
	\end{align*} 

Let us expand the left-hand side of this inequality: 
\begin{align*}
 &\sum_{n=1}^{\infty}\sup_{0< m <\infty}\bigg|\min\biggl\{\frac{1}{n}\:,\:\frac{1}{m}\biggr\}\sum_{k=1}^{m}\phi(k)\bigg|^p\\&=\sum_{n=1}^{\infty}\sup_{0< m <\infty}\bigg|\min\biggl\{\frac{1}{n}\:,\:\frac{1}{m}\biggr\}\sum_{k=1}^{m}\big(\psi(k)-\psi(k-1)\big)\bigg|^p\\&= \sum_{n=1}^{\infty}\sup_{0< m <\infty}\bigg|\min\biggl\{\frac{1}{n}\:,\:\frac{1}{m}\biggr\}\psi(m)\bigg|^p\\&=\sum_{n=1}^{\infty}\max\biggl\{\sup_{0< m \leq n}\frac{|\psi(m)|^p}{n^p},\: \sup_{n\leq  m <\infty} \frac{|\psi(m)|^p}{m^p}\biggr\}.
\end{align*}
This completes the proof of Theorem \ref{mainthm-2}. Now the inequality \eqref{sharper2} implies sharpness of the constant since $\big(\frac{p}{p-1}\big)^p$  is sharp for the inequality  \eqref{grad-hardy}.
\end{proof}

Now let us state an improvement of the inequality \eqref{best-hardy}.
\begin{theorem}\label{best-hardy-imp}
  Let $p\in\n$ with $p\ge 2$. Then for all $\psi\in C_c(\n)$ there holds
\begin{align}\label{eqn-best-hardy-imp}
     \sum_{l\in 2\n\cup\{0\}} c_{l}\sum_{n=1}^{\infty}\sup_{0< m <\infty}\bigg|\min\biggl\{\frac{1}{n^{\frac{l+p}{p}}},\frac{1}{m^{\frac{l+p}{p}}}\biggr\}\sum_{k=1}^{m}\psi(k)\bigg|^p \leq \sum_{n=1}^{\infty}|\psi(n)|^p,
\end{align}
where $c_{l}$ is given in \eqref{vpn}. 
\end{theorem}
\begin{proof}
Let us fix $p\in\n$ with $p\geq 2$ and $l\in 2\n\cup\{0\}$. Notice that $\frac{l+p}{p}\geq 1$, so for any $n\leq m$ we have
\begin{align}\label{m-1}
    \frac{m}{n}\leq \bigg(\frac{m}{n}\bigg)^{\frac{l+p}{p}}.
\end{align}
Therefore,  setting $w(n)=n^{\frac{l+p}{p}}$ in Proposition \ref{key-prop} we get
	\begin{align}\label{m-2}
	\sup_{0<m<\infty}\bigg|\min\biggl\{\frac{1}{n^{\frac{l+p}{p}}},\frac{1}{m^{\frac{l+p}{p}}}\biggr\}\sum_{k=1}^{m}\psi(k)\bigg|^p  \leq \left( \frac{1}{n^{\frac{l+p}{p}}}\sum_{k=1}^{n}|\Tilde{\psi}(k)|\right)^{p}
	\end{align}
Multiplying by nonnegative constant $c_{l}$ to both sides of \eqref{m-2} and taking summation over $n$, we deduce
	\begin{align}\label{m-3}
	    c_{l}\sum_{n=1}^{\infty}\sup_{0<m<\infty}\bigg|\min\biggl\{\frac{1}{n^{\frac{l+p}{p}}},\frac{1}{m^{\frac{l+p}{p}}}\biggr\}\sum_{k=1}^{m}\psi(k)\bigg|^p\leq c_{l}\sum_{n=1}^{\infty} \frac{1}{n^{l+p}}\bigg|\sum_{k=1}^{n}|\Tilde{\psi}(k)|\bigg|^{p}.
	\end{align}
Again taking summation over $l$ and using \eqref{best-hardy} for $\tilde{\psi}$, we obtain
\begin{align*}
    &\sum_{l\in 2\n\cup\{0\}} c_{l}\sum_{n=1}^{\infty}\sup_{0< m <\infty}\bigg|\min\biggl\{\frac{1}{n^{\frac{l+p}{p}}},\frac{1}{m^{\frac{l+p}{p}}}\biggr\}\sum_{k=1}^{m}\psi(k)\bigg|^p\\&\leq \sum_{l\in 2\n\cup\{0\}}c_{l}\sum_{n=1}^{\infty}\frac{1}{n^{l+p}}\bigg|\sum_{k=1}^{n}|\tilde{\psi}(k)|\bigg|^p\\&\overset{\eqref{best-hardy}}{\leq}\sum_{n=1}^{\infty}|\tilde{\psi}(n)|^p\overset{\eqref{norm}}{=}\sum_{n=1}^{\infty}|\psi(n)|^p.
\end{align*}
It completes the proof.
\end{proof}

Directly following the idea of the proof of Theorem \ref{mainthm-2}, a version of  Theorem \ref{best-hardy-imp} can be obtained. We state it without proof.
\begin{theorem}\label{mainthm-3}
 Let $p\in\n$ with $p\ge 2$. Then for all $\psi\in C_c(\n)$ with $\psi(0)=0$ there holds
\begin{align}\label{eqn-mainthm-3}
     \sum_{l\in 2\n\cup\{0\}} c_{l}\sum_{n=1}^{\infty}\max\biggl\{\sup_{0< m \leq n}\frac{|\psi(m)|^p}{n^{l+p}},\: \sup_{n\leq  m <\infty} \frac{|\psi(m)|^p}{m^{l+p}}\biggr\} \leq
     \sum_{n=1}^{\infty}|\nabla \psi(n)|^p,
\end{align}
where $c_{l}$ is given in \eqref{vpn}.   
\end{theorem}

\section{Discrete uncertainty principle}\label{4}
In this section, we discuss some immediate consequences of the obtained improved discrete Hardy inequality. In quantum mechanics, the uncertainty principle is a fundamental concept. It states that there is a limit to the precision with which certain pairs of physical properties, such as position and momentum, can be simultaneously known. The most well-known formulation of the uncertainty principle is probably the Heisenberg-Pauli-Weyl uncertainty principle, which is (in the discrete setting) one of the building blocks of condensed matter physics. Here we present another form of the uncertainty principle for 1D lattice. 
\begin{theorem}\label{hpw}
Let $1<p,q<\infty$ and  $1/p+1/q=1$. Then for any $\psi\in C_c(\n)$ with $\psi(0)=0$ there holds 
  \begin{align*}
     &\bigg(\frac{p-1}{p}\bigg)\max\biggl\{\sum_{n=1}^{\infty}\sup_{0< m \leq n}|\psi(m)|^p,\: \sum_{n=1}^{\infty}\sup_{n\leq  m <\infty} |\psi(m)|^p\biggr\}\\&\leq ||\nabla \psi(n)||_{\ell^p(\N)}\,\max\biggl\{\bigg(\sum_{n=1}^{\infty}\sup_{0< m \leq n}n^q|\psi(m)|^p\bigg)^{\frac{1}{q}}, \bigg(\sum_{n=1}^{\infty}\sup_{n\leq  m <\infty} {m^q}{|\psi(m)|^p}\bigg)^{\frac{1}{q}}\biggr\}.
 \end{align*}   
\end{theorem}
\begin{proof}
For $\psi\in C_c(\n)$ we compute
\begin{align*}
    &\sum_{n=1}^{\infty}\sup_{0< m \leq n}|\psi(m)|^p=\sum_{n=1}^\infty\bigg(\sup_{0< m \leq n}\frac{|\psi(m)|}{n}\bigg)\, \bigg(\sup_{0< m \leq n}n|\psi(m)|^{p-1}\bigg)\\&\leq \bigg(\sum_{n=1}^\infty \sup_{0< m \leq n}\frac{|\psi(m)|^p}{n^p}\bigg)^{\frac{1}{p}}\bigg(\sum_{n=1}^\infty \sup_{0< m \leq n}n^q|\psi(m)|^{p}\bigg)^{\frac{1}{q}}\\&\leq\bigg(\frac{p}{p-1}\bigg)\bigg(\sum_{n=1}^\infty |\nabla \psi(n)|^p\bigg)^{\frac{1}{p}}\bigg(\sum_{n=1}^\infty \sup_{0< m \leq n}n^q|\psi(m)|^{p}\bigg)^{\frac{1}{q}}.
\end{align*}
Similarly, we have
\begin{align*}
    &\sum_{n=1}^{\infty}\sup_{n\leq m <\infty}|\psi(m)|^p=\sum_{n=1}^\infty\bigg(\sup_{n\leq m <\infty}\frac{|\psi(m)|}{m}\bigg)\, \bigg(\sup_{n\leq m <\infty}m|\psi(m)|^{p-1}\bigg)\\&\leq \bigg(\sum_{n=1}^\infty \sup_{n\leq m <\infty}\frac{|\psi(m)|^p}{m^p}\bigg)^{\frac{1}{p}}\bigg(\sum_{n=1}^\infty \sup_{n\leq m <\infty}m^q|\psi(m)|^{p}\bigg)^{\frac{1}{q}}\\&\leq\bigg(\frac{p}{p-1}\bigg)\bigg(\sum_{n=1}^\infty |\nabla \psi(n)|^p\bigg)^{\frac{1}{p}}\bigg(\sum_{n=1}^\infty \sup_{n\leq  m <\infty}m^q|\psi(m)|^{p}\bigg)^{\frac{1}{q}}.
\end{align*}
In both the occasions, first the H\"older inequality is used and then Theorem \ref{mainthm-2} is applied. Finally, combining all together we arrive at the desired uncertainty principle.
\end{proof}

\section*{Acknowledgements} 
This research was funded by the Committee of Science of the Ministry of Science and Higher Education of Kazakhstan (Grant No. AP19674900). This work was also supported by Nazarbayev University grant 20122022FD4105. This project was discussed when the authors met at the Ghent Analysis \& PDE Center at Ghent University in Summer 2022. During the visit, the authors were supported by the FWO Odysseus 1 grant G.0H94.18N: Analysis and PDEs and by the Methusalem programme of the Ghent University Special Research Fund (BOF) (Grant number 01M01021). The authors would like to thank Prof. Michael Ruzhansky and the university for their support and hospitality.


\noindent


\begin{thebibliography}{100}
\bibitem{fish}F.~Fischer, \emph{A Non-Local Quasi-Linear Ground State Representation and Criticality Theory}, Calc. Var. Partial Differential Equations 62 (2023), no. 5, Paper No. 163, 33 pp, MR\href{https://doi.org/10.1007/s00526-023-02496-5}{4597627}.

\bibitem{fkp} F.~Fischer, M.~Keller, F.~Pogorzelski, \emph{An Improved Discrete $p$-Hardy Inequality}, Integral Equations Operator Theory 95 (2023), no. 4, Paper No. 24, 17 pp, MR\href{https://doi.org/10.1007/s00020-023-02743-6}{4648598}.

\bibitem{rupert-22} R.~L.~Frank, A.~Laptev, T.~Weidl, \emph{An improved one-dimensional Hardy inequality}, J. Math. Sci. (N.Y.) 268 (2022), no. 3, 323–342, MR\href{https://doi.org/10.1007/s10958-022-06199-8}{4533300}.

\bibitem{GKS} B.~Gerhat, D.~Krej\v{c}i\v{r}\'{\i}k, F.~\v{S}tampach, \emph{An improved discrete Rellich inequality on the half-line}, (2022), to appear in Israel
J. Math., arXiv:\href{https://doi.org/10.48550/arXiv.2206.11007}{2206.11007}.

\bibitem{gupta1d} S.~Gupta, \emph{Discrete weighted Hardy inequality in 1-D}, J. Math. Anal. Appl. 514 (2022), no. 2, Paper No. 126345, 19 pp, MR\href{https://doi.org/10.1016/j.jmaa.2022.126345}{4426119}.

\bibitem{gupta-sym} S.~Gupta, \emph{Symmetrization inequalities on one-dimensional integer lattice}, (2022), arXiv:\href{https://arxiv.org/abs/2204.11647v1}{2204.11647}.

\bibitem{hardy} G.~H.~Hardy, \emph{Note on a theorem of Hilbert}, Math. Z. 6 (1920), no. 3-4, 314–317, MR\href{https://doi.org/10.1007/BF01199965}{1544414}.

\bibitem{hlp} G.~H.~Hardy, J.~E.~Littlewood, G.~Pólya, \emph{Inequalities. 2d ed.}, Cambridge, at the University Press, (1952), MR\href{https://mathscinet.ams.org/mathscinet/search/publdoc.html?agg_year_1952=1952&batch_title=Selected\%20Matches\%20for\%3A\%20Items\%20authored\%20by\%20Hardy\%2C\%20Godfrey\%20Harold&fmt=doc&pg1=INDI&s1=81295&searchin=&sort=newest&vfpref=html&r=2&mx-pid=46395}{0046395}. 

\bibitem{hng}  Y.~C.~Huang, \emph{A first order proof of the improved discrete Hardy inequality}, Arch. Math. (Basel) 117 (2021), no. 6, 671–674, MR\href{https://doi.org/10.1007/s00013-021-01653-6}{4340887}.

\bibitem{ll} L.~Kapitanski, A.~Laptev, \emph{On continuous and discrete Hardy inequalities}, J. Spectr. Theory 6 (2016), no. 4, 837–858, MR\href{https://doi.org/10.4171/JST/144}{3584186}.

\bibitem{kpm-2} M.~Keller, Y.~Pinchover, F.~Pogorzelski, \emph{An improved discrete Hardy inequality}, Amer. Math. Monthly 125 (2018), no. 4, 347–350, MR\href{https://doi.org/10.1080/00029890.2018.1420995}{3779222}.

\bibitem{kpm-3} M.~Keller, Y.~Pinchover, F.~Pogorzelski, \emph{Criticality theory for Schr{\"o}inger operators on graphs}, J. Spectr. Theory 10 (2020), no. 1, 73–114, MR\href{https://doi.org/10.4171/JST/286}{4071333}.

\bibitem{kpm-2021} M.~Keller, Y.~Pinchover, F.~Pogorzelski, \emph{From Hardy to Rellich inequalities on graphs}, Proc. London Math. Soc. 122 (2021), no. 3, 458–477, MR\href{https://doi.org/10.1112/plms.12376}{4230061}.

\bibitem{kpm-1} M.~Keller, Y.~Pinchover, F.~Pogorzelski, \emph{Optimal Hardy inequalities for Schr{\"o}inger operators on graphs}, Comm. Math. Phys. 358 (2018), no. 2, 767–790, MR\href{https://doi.org/10.1007/s00220-018-3107-y}{3774437}.

\bibitem{kls22}D.~Krej\v{c}i\v{r}\'{\i}k,
A.~Laptev, F.~\v{S}tampach, \emph{Spectral enclosures and stability for non-self-adjoint discrete Schrödinger operators on the half-line}, Bull. Lond. Math. Soc. 54 (2022), no. 6, 2379–2403, MR\href{https://doi.org/10.1112/blms.12700}{4549127}.

\bibitem{kjc}D.~Krej\v{c}i\v{r}\'{\i}k, F.~\v{S}tampach, \emph{A sharp form of the discrete Hardy inequality and the Keller-Pinchover-Pogorzelski inequality}, Amer. Math. Monthly 129 (2022), no. 3, 281–283, MR\href{https://doi.org/10.1080/00029890.2022.2011569}{4399059}.
	
\bibitem{kuf} A.~Kufner, L.~Maligranda, L.~E.~Persson, \emph{The Prehistory of the Hardy Inequality}, Amer. Math. Monthly 113 (2006), no. 8, 715–732, MR\href{https://doi.org/10.2307/27642033}{2256532}.

\bibitem{lef-1} P.~Lef{\'e}vre, \emph{A short direct proof of the discrete Hardy inequality}, Arch. Math. (Basel) 114 (2020), no. 2, 195–198, MR\href{https://doi.org/10.1007/s00013-019-01395-6}{4055148}.

\bibitem{lef-2} P.~Lef{\'e}vre, \emph{Weighted discrete Hardy's inequalities}, Ukrainian Math. J. 75 (2023), no. 7, 1153–1157, MR\href{https://mathscinet.ams.org/mathscinet/article?mr=4679549}{4679549}.

\bibitem{lif} E.~Liflyand, \emph{Harmonic analysis on the real line—a path in the theory}, Pathways in Mathematics. Birkhäuser/Springer, Cham, (2021), MR\href{https://doi.org/10.1007/978-3-030-81892-0}{4369961}.

\bibitem{rrs-22} P.~Roychowdhury, M.~Ruzhansky, D.~Suragan, \emph{Multidimensional Frank-Laptev-Weidl improvement of the Hardy Inequality}, Proc. Edinb. Math. Soc. (2) 67 (2024), no. 1, 151–167, MR\href{https://doi.org/10.1017/s0013091523000780}{4713034}.
\end{thebibliography}
\end{document}